\documentclass[11pt,letterpaper]{article}
\usepackage[utf8]{inputenc}
\usepackage{tikz}
\usepackage{mathrsfs}
\usepackage{cite}
\usepackage[small,nohug,heads=LaTeX]{diagrams}
\usepackage{amsmath}
\usepackage{mathrsfs}
\usepackage{hyperref}
\usepackage{amsfonts}
\usepackage{amsthm}
\usepackage{amssymb}
\newtheorem{define}{Definition}[section]

\theoremstyle{remark}

\theoremstyle{plain}
\newtheorem{theo}[define]{Theorem}
\newtheorem{lemma}[define]{Lemma}

\newtheorem{cor}[define]{Corollary}
\usepackage{amssymb}
\usepackage[left=2cm,right=2cm,top=2cm,bottom=2cm]{geometry}
\author{Michael Lambert}
\title{Representation Embeddings of Cartesian Theories}
\begin{document}
\maketitle

\begin{abstract}
A representation embedding between cartesian theories can be defined to be a functor between the respective categories of models that preserves finitely-generated projective models and that preserves and reflects certain epimorphisms.  This recalls standard definitions in the representation theory of associative algebras.  The main result of this paper is that a representation embedding in the general sense preserves undecidability of theories; that is, if 
\[ E\colon \mathbb T_1\text -\mathbf{Mod}(\mathbf{Set})\to\mathbb T_2\text{-}\mathbf{Mod}(\mathbf{Set})
\]
is a representation embedding of cartesian theories, then if $\mathbb T_1$ is undecidable, so is $\mathbb T_2$.  This result is applied to obtain an affirmative resolution of a reformulation in cartesian logic of a conjecture of M. Prest that every wild algebra over an algebraically closed field has an undecidable theory of modules.
\end{abstract}

\tableofcontents

\section{Introduction} 

\subsection{The ``Wild Implies Undecidable'' Conjecture}

Within the cartesian fragment of first-order categorical logic, this paper presents a reformulation of the conjecture of M. Prest that any wild algebra has an undecidable theory of modules \cite{Prest}.  The background  for the reformulation is explained over the course of \S 2.  An affirmative resolution of the reformulated conjecture and an application to the original appear in \S4.  The extent to which this ought to be seen as a confirmation of the original conjecture is discussed in $\S$5.

Let $k$ denote an algebraically closed field.  A $k$-algebra is a ring with a compatible $k$-vector space structure.  A module over a $k$-algebra $A$ is a $k$-vector space $M$ with an action of $A$ that is compatible with the group and vector space structures on $M$.  Consider the free $k$-algebra on two symbols, $k\langle X,Y\rangle$.  The classical, first-order theory of $k\langle X,Y\rangle$-modules is known to admit no Turing Machine algorithm that will establish whether a given sentence of the theory is a theorem \cite{Prest},\cite{Baur}.  In this sense the theory is undecidable.  Now, a $k$-algebra $S$ is wild if its category of modules admits a representation embedding.  This is a finitely generated $(S,k\langle X,Y\rangle)$-bimodule $M$, free over $k\langle X,Y\rangle$, such that an induced functor $M\otimes_{k\langle X,Y\rangle} - \colon k\langle X,Y\rangle$-$\mathrm{mod}\to S$-$\mathrm{mod}$, between categories of finite-dimensional modules, preserves and reflects indecomposability and isomorphism \cite{Prest}, \cite{Benson}.  The conjecture is that any finite-dimensional wild algebra has an undecidable theory of modules. 

The conjecture is known to be true for certain classes of algebras (for example, for path algebras of quivers without relations and for controlled algebras as in \cite{GP} which lists further references).  In outline, an outright proof of the conjecture would consist in somehow using the semantic relationship given by $M\otimes_{k\langle X,Y\rangle}-$ to induce a relationship between the syntactical systems.  This could be a translation of theories.  Ordinarily, a translation of theories induces a relationship between categories of models.  Thus, a proof of the conjecture as just outlined asks for precisely the reverse.  The proof given here uses the syntactic categories of categorical logic to accomplish this.  Roughly, these categories are algebraizations of syntax that classify models in suitably structured categories.  However, to use this technology and its completeness theorems, the algebraic data of $k\langle X,Y\rangle$-modules will be studied in its formulation in the categorical semantics of cartesian logic, rather than in classical first-order model theory.  

This underlying idea requires no special properties of the module categories.  Thus, in \S 4 a representation embedding of cartesian theories is defined to be a functor of categories of models preserving finitely-generated projective models and preserving and reflecting certain epimorphisms.  This is justified by Theorem 4.3 which characterizes finitely-generated projectives in certain categories of cartesian set-valued functors.  The main result, Theorem 4.5, shows that any representation embedding roughly amounts to a conservative translation of theories and thus preserves undecidability.  It is from this general result that a proof of the reformulated ``wild implies undecidable'' conjecture is obtained as a corollary.  It is not claimed, however, that this result resolves the original conjecture.  It is only asserted that the categorical reformulation clarifies the problem and provides a partial answer.  

\subsection{Background Category Theory} 

The references for all category theory are \cite{MacLane}, \cite{Handbook1} and A1 of \cite{elephant}.  In the interest of making the paper more-or-less self-contained, here are summarized some of the main points and notations.

Throughout $\mathbf{Set}$ denotes the category of sets and set functions; $\mathbf{Cat}$ is the category of small categories and their functors; and $\mathbf{Grph}$ is the category of directed graphs.  These and all other categories are locally small.  In general categories will be denoted with script capitals.  The set of morphisms between objects $c$ and $d$ of a locally small category $\mathscr X$ will be denoted $\mathscr X(c,d)$, or perhaps $\mathrm{Hom}(c,d)$, if $\mathscr X$ is clear.

If $\mathscr C$ denotes a small category, then $[\mathscr C,\mathbf{Set}]$ is the category of functors, or copresheavs, $\mathscr C\to \mathbf{Set}$ and their natural transformations.  This functor category is complete, cocomplete, and every morphism has a canonical epi-mono image factorization.  Limits and colimits are computed pointwise in $\mathbf{Set}$.  A functor $F\colon \mathscr C\to \mathbf{Set}$ is representable if there is an object $c\in\mathscr C_0$ for which there is a natural isomorphism $F\cong \mathscr C(c,-)$ with the canonical representable functor $\mathbf y(c) = \mathscr C(c,-)$.

Yoneda's Lemma gives a parameterization of natural transformations from a representable functor to a $\mathbf{Set}$-valued functor. It states that for any $c\in \mathscr C_0$ and functor $K\colon\mathscr C\to\mathbf{Set}$, there is a bijection 
\[\mathrm{Nat}(\mathbf y(c),K) \cong Kc
\]
natural in $c$ and $K$.  Now, $\mathbf y$ is part of a functor.  For any $h\colon c\to d$ of $\mathscr C$, there is a natural transformation $\mathscr C(h,-)\colon\mathscr C(d, -)\to\mathscr C(c,-)$ taking $f\colon d\to c'$ to its pullback by $h$, namely, $hf\colon c\to c'$.  The traditional super-scripted `$*$' denotes these maps, i.e. $h^*(f) = hf$.  The Yoneda embedding $\mathbf y\colon \mathscr C^{op} \to [\mathscr C,\mathbf{Set}]$ is then given by $c\mapsto \mathscr C(c,-)$ on objects and $h\mapsto h^*$ on arrows.  It is fully faithful and injective on objects. 

The importance of representable functors is partly that together they form a dense class of projective generators for $[\mathscr C,\mathbf{Set}]$.  This means that representables are projective; and that every $K\colon \mathscr C\to\mathbf{Set}$ is canonically a colimit of representables.  The indexing category is the category of elements, denoted
\[ \int_{\mathscr C}K.
\]
It has as objects those $(c,x)$ with $x\in Kc$ and as arrows those $f\colon c\to d$ such that $Kf(x)=y$.  The Yoneda embedding composed with the projection from the category of elements yields a diagram
\[ D\colon (\int_{\mathscr C}K)^{op}\longrightarrow [\mathscr C,\mathbf{Set}].
\]
The colimit in $[\mathscr C,\mathbf{Set}]$ is $P$ itself.  For $P$ is the vertex of a universal cocone.  That is, for each $(c,x)$ of the category of elements, there is a unique natural transformation 
\[ \tilde x\colon \mathbf y(c)\to P
\] 
given by the Yoneda isomorphism.  The family of such $\tilde x$ yields a cocone with vertex $P$ by naturality.  Universality is established by further use of the Yoneda isomorphism and its naturality.  

A category is cartesian if it has all finite limits, equivalently, finite products and equalizers.  A functor is cartesian if it preserves finite limits.  When $\mathscr C$ is cartesian, $K$ is cartesian if, and only if, the colimit above is filtered, as in 6.1.2 of \cite{Handbook1}.  Thus, every cartesian $\mathbf{Set}$-valued functor on a cartesian category is a filtered colimit of representable functors.

When $\mathscr C$ is a small cartesian category, $\mathbf{Cart}(\mathscr C,\mathbf{Set})$ denotes the full subcategory of $[\mathscr C,\mathbf{Set}]$ of cartesian copresheaves $\mathscr C\to\mathbf{Set}$.  Limits and filtered colimits of $\mathbf{Cart}(\mathscr C,\mathbf{Set})$ are simply inherited from $[\mathscr C,\mathbf{Set}]$.  General colimits exist, but have a separate construction arising roughly from a left-adjoint to the canonical inclusion of $\mathbf{Cart}(\mathscr C,\mathbf{Set})$ as a subcategory.  In addition $\mathbf{Cart}(\mathscr C,\mathbf{Set})$ inherits epi-mono factorizations from the functor category.  In particular canonical image copresheaves of morphisms in $\mathbf{Cart}(\mathscr C,\mathbf{Set})$ are again cartesian.  Each representable copresheaf $\mathscr C\to\mathbf{Set}$ is cartesian.  Thus, the Yoneda embedding factors through $\mathbf{Cart}(\mathscr C,\mathbf{Set})$.

Roughly, cartesian functors $M\colon \mathscr C\to\mathbf{Set}$ are models or algebras for an equational theory determined by $\mathscr C$.  The important property of being finitely generated has a purely category-theoretic definition, stated here formally.  In general a filtered union is a colimit of a diagram $D\colon \mathscr J\to\mathscr X$ where $\mathscr J$ is filtered and the image under $D$ of any morphism of $\mathscr J$ is a monomorphism in $\mathscr X$.  Colimit preservation means that the canonical map
\[ \lim_{\to}[\mathscr X,\mathbf{Set}](a,d_j)\longrightarrow [\mathscr X,\mathbf{Set}](a,\lim_{\to}d_j)
\]
is an isomorphism of sets.  

\begin{define}  An object $a \in\mathscr X_0$ of a category $\mathscr X$ is finitely generated if the representable functor 
\[\mathscr X(a,-)\colon\mathscr X\to\mathbf{Set}
\]
preserves filtered unions.  
\end{define}

That the definition specializes to the usual one for concrete models of usual algebraic theories is established in 5.22 of \cite{AR}.  In general, representable functors are finitely generated in $[\mathscr C,\mathbf{Set}]$ and in $\mathbf{Cart}(\mathscr C,\mathbf{Set})$.  The full subcategory of finitely-generated objects of a category $\mathscr X$ will be denoted with a subscript $\mathscr X_{fg}$.  In particular $\mathbf{Cart}(\mathscr C_{\mathbb T},\mathbf{Set})_{fg}$ for a cartesian syntactic category $\mathscr C_{\mathbb T}$ will be equivalent to the category of finitely-generated models of a cartesian theory $\mathbb T$. 

\section{First-Order Categorical Logic}

\subsection{Algebraic and Cartesian Theories} 

The presentation of first-order categorical logic in chapter D1 of \cite{elephant} will be adopted.  One notational change is that lists of variables are denoted using boldface `$\mathbf x$' instead of an over-arrow.  Other unexplained notational changes should be clear from context.  For any theory there should be countably many variables of each type.  In addition, for each fixed type, variables come with a total ordering.  An expression such as `$\mathbf x.\phi$' denotes a formula-in-context.  The context `$\mathbf x$' can be taken to be minimal, in that only the free variables of the formula appear. 

A theory $\mathbb T$ is given with a signature $\Sigma$ and includes both its axioms and the consequences obtained under a specified notion of derivability.  A theory is Horn if these involve only Horn formulas, that is, those built from the atomic ones using only `$\top$' and `$\wedge$'.  A regular theory is a Horn theory allowing `$\exists$'.  For a regular theory $\mathbb T$, the class of cartesian formula relative to $\mathbb T$ is defined inductively.  Atomic formulas of $\mathbb T$ are cartesian, finite conjunctions of cartesian formulas are cartesian, and $\mathbf x.\exists y\phi$ is cartesian if $\mathbf x,y.\phi$ is cartesian and $\phi\wedge \phi[z/y]\vdash y=z$ is provable in $\mathbb T$.  A cartesian theory is any regular theory admitting a partial-ordering of its axioms such that any given axiom is cartesian relative to the sub-theory generated by those axioms preceding it in the order.  For example, the theory of categories is cartesian.

An algebraic theory is a Horn theory, having a single sort, with function symbols but no relation symbols, and all of whose axioms are of the form `$\top\vdash_{\mathbf{x}} \phi$' where $\phi$ is Horn.  The theories of groups, abelian groups, rings, modules over a fixed ring, and lattices are all algebraic.  An axiomatization of the algebraic theory of modules over a ring $R$ is as follows.  The signature has one sort $A$, a binary function symbol $+\colon A\times A\to A$, a unit symbol $0\colon 1\to A$, an inverse unary function symbol $-(-)\colon A\to A$, and a family of unary function symbols $r\colon A\to A$, indexed by $r\in R$, all satisfying the group axioms,
\begin{align}
\top &\vdash_{x,y,z} ((x+y) + z) = (x+(y+z)) \notag\\ 
\top &\vdash_{x,y} (x+y) = (y+x) \notag\\
\top &\vdash_x (x+0)= x\notag \\
\top &\vdash_x (x+(-x)) = 0\notag
\end{align}
and the action axioms
\begin{align}
\top &\vdash_{x} rs(x) = r(s(x)) \notag\\ 
\top &\vdash_x 1(x)=x \notag\\
\top &\vdash_x (r+s)x = r(x)+s(x)\notag \\
\top &\vdash_{x,y} r(x+y) = r(x)+r(y) \notag.
\end{align}
If further $R$ should be an algebra and its modules vector spaces, then one would add to the signature more unary function symbols $A\to A$ indexed by $k$ with appropriate axioms giving a compatible action.  Notice that this algebraic theory is not Boolean, in that `$\top\vdash \phi\vee\neg\phi$' is not formulable in cartesian logic. 

In order to use the technology of syntactic categories, the cartesian theory of $R$-modules will be studied instead of the algebraic theory.  Denote this theory by $\mathbb T_R$.  It has the same signature and axioms as the algebraic theory, but includes all the consequences obtained from cartesian derivations.  The cartesian theory will be thought of as ``algebraic'' in nature, since its axioms are those of the algebraic theory.  Accordingly, the theory is an example of what is called an ``essentially algebraic'' theory.

For any theory $\mathbb T$, the precise definition of ``structure'' and ``model'' in an apropriate category $\mathscr D$ is given in D1.2 of \cite{elephant}.  Roughly, a model $\mathbf M\colon \mathbb T\to\mathscr D$ is a structure that validates the axioms of the theory.  The category of models and their morphisms in a suitable category $\mathscr D$ is denoted $\mathbb T$-$\mathbf{Mod}(\mathscr D)$.  For any ring $R$, the category of $\mathbf{Set}$-models of the algebraic theory of $R$-modules is isomorphic to $R$-$\mathbf{Mod}$.

Syntactic categories are a major part of subsequent proofs.  In view of the importance of existential quantification in their construction, the following lemma is proved here, establishing an expected useful property.  The proof mimics that in the logic of \cite{Bell}, but uses the system of the main reference.

\begin{lemma}  There is the following derived rule for existential introduction, namely, 
\[ \dfrac{\phi\vdash_{\mathbf{x}} \psi(\tau/y)}{\phi\vdash_{\mathbf{x}} \exists y\psi}
\]
where $\tau$ is free for $y$ in $\psi$ and $\mathbf{x}$ is a suitable context for $\phi$ and $\psi(\tau/y)$.
\begin{proof}  The derivation (with omitted context notation)
\[\dfrac{\dfrac{}{\phi\vdash \psi(\tau/y)}\qquad \dfrac{\dfrac{ \dfrac{}{\exists y\psi\vdash  \exists y\psi}}{\psi\vdash \exists y\psi}}{\psi(\tau/y)\vdash \exists y\psi}}{\phi\vdash \exists y\psi}
\]
uses from top to bottom, $\exists$-elimination, substitution, and the cut rule, all detailed in D1.3.1 of \cite{elephant}.  \end{proof}  
\end{lemma}

\subsection{Syntactic Categories}

The syntactic category of a cartesian theory $\mathbb T$, described in D1.4 of \cite{elephant}, is denoted $\mathscr C_{\mathbb T}$.  It has as objects $\alpha$-equivalence classes of formulas-in-context $\lbrace \mathbf{x}.\phi\rbrace$ where $\mathbf x.\phi$ is cartesian relative to $\mathbb T$; and as arrows those $\mathbb T$-provable equivalence classes $[\theta]\colon\lbrace \mathbf{x}.\phi\rbrace\to \lbrace \mathbf{y}.\psi\rbrace$ of formulas $\theta$, cartesian relative to $\mathbb T$, for which the $\mathbb T$-provably functional sequents 
\begin{align} \theta &\vdash_{\mathbf{x},\mathbf{y}}\phi\wedge \psi\notag \\
\theta\wedge\theta[\mathbf{z}/\mathbf{y}]&\vdash_{\mathbf{x},\mathbf{y},\mathbf{z}} \mathbf{y} = \mathbf{z} \notag\\
\phi &\vdash_{\mathbf{x}}\exists\mathbf{y}\theta\notag
\end{align}
are derivable.  Composition of $[\theta]\colon\lbrace \mathbf{x}.\phi\rbrace\to \lbrace \mathbf{y}.\psi\rbrace$ and $[\gamma]\colon\lbrace \mathbf{y}.\psi\rbrace\to \lbrace \mathbf{z}.\chi\rbrace$ is defined as $[\gamma][\theta] := [\exists y(\theta\wedge \gamma)]$, and identity $\lbrace\mathbf x.\phi\rbrace\to\lbrace\mathbf x'.\phi[\mathbf x'/\mathbf x]\rbrace$ is defined as $[\phi\wedge\mathbf x=\mathbf x']$.  It should be noted that, strictly speaking, no algebraic theory has a syntactic category.  One is obtained by taking the cartesian theory with the same signature and axioms.  This is what was meant in the discussion of the cartesian theory $\mathbb T_R$ in $\S$2.1. 

Recall from D1.4.2 that $\mathscr C_{\mathbb T}$ is cartesian.  The proof there gives explicit constructions of products and equalizers.  Let $\mathrm{Sub}(\lbrace \mathbf{x}.\phi\rbrace)$ denote the poset of subobjects of the object $\lbrace \mathbf{x}.\phi\rbrace$ of $\mathscr C_{\mathbb T}$.  This is a meet semi-lattice.  The following result shows that $\mathscr C_{\mathbb T}$ classifies models in any cartesian category.

\begin{lemma}  For any cartesian theory $\mathbb T$ and any cartesian category $\mathscr D$, there is an equivalence 
\[\mathbf{Cart}(\mathscr C_{\mathbb T},\mathscr D)\simeq \mathbb T\text{-}\mathbf{Mod}(\mathscr D).
\]
\begin{proof} See \cite{elephant} D1.4.7 for the functors and a proof.\end{proof}
\end{lemma}

The proof of the lemma shows that every $\mathscr D$-model $\mathbf M\colon\mathbb T\to \mathscr D$ is the image of a certain ``universal model'' $\mathbf M_\mathbb{T}$ under a cartesian functor $F\colon  \mathscr C_{\mathbb T}\to \mathscr D$.  The model $\mathbf M_\mathbb{T}$ arises from a canonical interpretation of $\mathbb T$ in $\mathscr C_\mathbb{T}$.  Each sort $A$ is interpreted as the object $\lbrace x.\top\rbrace$ with $x$ of type $A$, and each function symbol $f\colon A_1,\dots, A_n\to B$ is interpreted as the morphism
\[[f(x_1,\dots, x_n)=y]\colon \lbrace x_1,\dots, x_n.\top\rbrace \to\lbrace y.\top\rbrace.
\]
The axiomatization of the modules of some fixed ring $R$ includes distinguished unary function symbols indexed by the elements $r\in R$.  At least in this unary case, composition of morphisms under canonical interpretation as in the display behaves as one expects.

\begin{lemma}  Let $f$ and $g$ denote two unary function symbols of a cartesian theory $\mathbb T$ with signatures $f\colon A\to B$ and $g\colon B\to C$.  The composition of the canonical interpretations is then
\[ [g(y)=z][f(x) =y] = [g(f(x)) = z]\colon \lbrace x.\top\rbrace\to\lbrace z.\top\rbrace.
\]
\end{lemma}
\begin{proof}  From the definition of composition, one has that the left hand side of the equality in the display is the class of $\exists y(g(y)=z\wedge f(x) = y)$.  But this is provably equivalent to $g(f(x)) = z$ by Lemma 2.1 above.  Hence the two formulas represent the same class and the equality holds.\end{proof}

The presence of the universal model $\mathbf M_{\mathbb T}$ leads to a completeness theorem for cartesian logic, following from the remarkable result that provability in $\mathbb T$ is equivalent to satisfaction in $M_\mathbb{T}$.

\begin{lemma}  For a cartesian theory $\mathbb T$, a cartesian sequent $\phi\vdash_{\mathbf{x}}\psi$ is provable in $\mathbb T$ if, and only if, it is satisfied in $M_\mathbb T$.  In addition, $\phi\vdash_{\mathbf{x}}\psi$ is provable in $\mathbb T$ if, and only if, $\lbrace \mathbf{x}.\phi\rbrace\leq \lbrace \mathbf{x}.\psi\rbrace$ holds as subobjects of $\lbrace\mathbf{x}.\top\rbrace$, that is, if, and only if, there is a monic arrow $\lbrace \mathbf{x}.\phi\rbrace\to \lbrace \mathbf{x}.\psi\rbrace$ of $\mathscr C_{\mathbb T}$.
\end{lemma}
\begin{proof}  See D1.4.4 and D1.4.5 of \cite{elephant}.\end{proof} 

\subsection{Translations Between Cartesian Theories}

Throughout this subsection $\mathbb T_1$ and $\mathbb T_2$ denote cartesian theories.  Motivated by \cite{Bell}, a translation $t\colon \mathbb T_1\to\mathbb T_2$ is a function of signatures $t\colon\Sigma_1\to\Sigma_2$ sending types $A$ to types $t(A)$ and function symbols $f\colon A\to B$ to function symbols $t(f)\colon t(A)\to t(B)$, and similarly for relation symbols, that preserves provability in a sense to be spelled out presently.  The assignment between signatures is inductively extended to all terms and formulas of $\mathbb T_1$ in the following manner.  For a variable $x$ of type $A$, say, $t(x)$ is the variable of type $t(A)$ having the same place in the ordering of the variables of type $t(A)$.  The subsequent definitions involved in this extension are
\begin{align} t(\top) &:=\top \notag\\
t(f(\tau)) &:=t(f)(t(\tau)) \notag \\
t(\tau =\sigma)&:=t(\tau) = t(\sigma) \notag\\
t(\phi\wedge\psi) &:=t(\phi)\wedge t(\psi) \notag \\
t(\exists x\phi)&:=\exists y t(\phi) \notag
\end{align}
where $\sigma$ and $\tau$ are terms, $\phi$ and $\psi$ are formulas, and $y$ is a variable free in $t(\phi)$ and of the type corresponding to that of $x$ under $t$.  Thus, the function $t\colon \Sigma_1\to\Sigma_2$ as above, suitably extended to all terms and formulas of $\mathbb T_1$, is a translation of these cartesian theories if whenever $\phi\vdash\psi$ is provable in $\mathbb T_1$, then $t(\phi)\vdash t(\psi)$ is provable in $\mathbb T_2$.  Use the notation $t\colon\mathbb T_1\to\mathbb T_2$ for such a translation.  A translation is conservative if it reflects provability in the sense that if $t(\phi)\vdash t(\psi)$ is provable in $\mathbb T_1$, then $\phi\vdash\psi$ is provable in $\mathbb T_2$.  

\begin{lemma}  A translation $t\colon\mathbb T_1\to\mathbb T_2$ of cartesian theories induces a model $\mathbf T\colon \mathbb T_1\to\mathscr C_{\mathbb T_2}$, hence a cartesian functor $T\colon\mathscr C_{\mathbb T_1}\to\mathscr C_{\mathbb T_2}$.  If $t$ is conservative, then so is $T$.  
\begin{proof}  A translation $t\colon \mathbb T_1\to\mathbb T_2$ amounts to a structure $\mathbf T\colon \mathbb T_1\to \mathscr C_{\mathbb T_2}$.  In fact $\mathbf T$ is $t$ followed by the universal model $\mathbf M_{\mathbb T_2}$.  So, that $\mathbf T$ is a model follows from the fact that $t$ preserves provability.  Models and cartesian functors correspond as in the equivalence of Lemma 2.2. The definition of the functors in the equivalence in fact show that cartesian functors between the syntactic categories are in one-to-one correspondence with $\mathbb T_1$-models in $\mathscr C_{\mathbb T_2}$.  \end{proof}
\end{lemma}

There is a category $\mathbf{CartThy}$ of cartesian theories and their translations.  The above lemma amounts to showing that there is a functor $\mathbf{CartThy}\to\mathbf{Cart}$ to the category of small cartesian categories and cartesian functors.  The question of whether this is at least one half of an adjunction concerns internal languages of cartesian categories and is of interest but not needed in this paper.  A translation could also be defined as a functor of syntactic categories $T\colon\mathscr C_{\mathbb T_1}\to\mathscr C_{\mathbb T_2}$ that preserves monomorphisms, hence provability of sequents by Lemma 2.4.  Accordingly such a translation would be called conservative if it reflected monomorphisms, hence provability.  It should be noted, however, that this functor does not necessarily induce a function of signatures.  So, there is no purely syntactical relationship in this definition of translation.

Let $\mathbf x.\phi$ denote any formula-in-context of $\mathbb T$.  Let $\mathbb T[\phi]$ denote the cartesian theory obtained from $\mathbb T$ by adding $\top\vdash_{\mathbf x}\phi$ as an axiom.  Think of $\mathbb T[\phi]$ as obtained from $\mathbb T$ by ``adjoining'' the assumption $\mathbf x.\phi$.  Now, there is an evident translation $t\colon \mathbb T\to \mathbb T[\phi]$ given by simply including types $A$ and function symbols $f\colon A\to B$ of $\mathbb T$ into $\mathbb T[\phi]$ and extending the assignment.  This is a translation as a proof of $\psi\vdash\chi$ in $\mathbb T$ is also a proof in $\mathbb T[\phi]$.  This translation amounts to an inclusion of syntactic categories $\iota\colon \mathscr C_{\mathbb T}\to \mathscr C_{\mathbb T[\phi]}$.  The following attests to the intuition that $\mathbb T[\phi]$ is the theory of those formulas entailed by $\phi$ in $\mathbb T$. 

\begin{lemma} In the above notation, $\top\vdash\psi$ is provable in $\mathbb T[\phi]$ if, and only if, $\phi\vdash\psi$ is provable in $\mathbb T$.
\begin{proof}  On the one hand, if $\top\vdash\psi$ is provable in $\mathbb T[\phi]$, then there is a proof of $\phi\vdash \psi$ in $\mathbb T$, obtained by regarding any usage of $\phi$ as an assumption.  On the other hand, if $\phi\vdash\psi$ is provable in $\mathbb T$, then since $\top\vdash \phi$ is an axiom of $\mathbb T[\phi]$, an application of the cut rule in $\mathbb T[\phi]$ proves $\top\vdash\psi$ in $\mathbb T[\phi]$.  \end{proof}
\end{lemma}

\section{An Undecidable Theory}

The classical first-order theory of modules over $k\langle X,Y\rangle$ is known to be undecidable \cite{Baur},\cite{Prest}.  That the cartesian theory $\mathbb T_{k\langle X,Y\rangle}$ developed above is also undecidable will be established presently.  For the proof, we need the fact that there is a $2$-generator finitely presented monoid with an undecidable word problem.  That there is such a thing has been established for example by Y. Matiyasevich \cite{Matiyasevich}.  Use the notation $N =\langle a,b \mid u_1=v_1,\dots, u_r=v_r\rangle$ for a fixed choice of one such monoid.

A word on how to view the elements and relations  of $N$ is necessary here.  That is, the identifications $u_1=v_1,\dots, u_r=v_r$ should be seen as rewrite rules.  In more detail, $u_1, \dots, u_r,v_1,\dots, v_r$ are words in the letters $a$ and $b$ viewed in the free monoid on the set $\lbrace a,b\rbrace$.  Thus, the expressions in terms of the letters $a$ and $b$ are unique.  The equations $u_i = v_i$ are really a shorthand for possible replacements.  That is, in any given word $u$ free over $\lbrace a,b\rbrace$ it is permitted to replace an occurrence of $u_i$ by $v_i$, or \emph{vice versa}, obtaining in either case what is considered to be an equivalent word over $\lbrace a,b\rbrace$.  Thus, words $u$ and $v$ over $\lbrace a,b\rbrace$ are considered equivalent if there is a finite sequence of replacements of this type by which either one is obtained from the other.  As this is an equivalence relation, $N$ with its presentation $N =\langle a,b \mid u_1=v_1,\dots, u_r=v_r\rangle$ is a quotient of the free monoid on $\lbrace a,b\rbrace$ modulo replacement.  And so, equality in $N$ is the same thing as equivalence in the free monoid on $\lbrace a,b\rbrace$ modulo replacement. 

By free generation in $k\langle X,Y\rangle$, the association $a\mapsto X$ and $b\mapsto Y$ induces a bijection between words over $\lbrace a, b\rbrace$ and the monomials in the ``letters'' $X$ and $Y$ of $k\langle X,Y\rangle$.  Use the notation $f_i$ and $g_i$ in $k\langle X,Y\rangle$ for the images of $u_i$ and $v_i$ respectively under this correspondence.  Thus, $f_i$ and $g_i$ do double duty as unary function symbols of $\mathbb T_{k\langle X,Y\rangle}$.  In what follows, $u$ and $v$ will denote arbitrary words of $N$.  In such instances, $f$ and $g$ will denote the corresponding images monomials in $k\langle X,Y\rangle$. 

\begin{lemma}  Consider the cartesian theory $\mathbb T_{k\langle X,Y\rangle}$.  For simplicity of notation denote this theory by $\mathbb T$.  In the notation of the preceding discussion let $\phi$ denote the formula
\[  \bigwedge_{j=1}^r (f_{i}(x) = g_{i}(x)).
\]
There is then a well-defined functor $\bar\rho\colon N\to \mathscr C_{\mathbb T[\phi]}$. 
\end{lemma}
\begin{proof} Recall from II.7 of \cite{MacLane} that there is a free-underlying adjunction $F\colon \mathbf{Grph} \rightleftarrows \mathbf{Cat}\colon U$ where $FG$ is the free category on the graph $G$.  Let $Q$ denote the figure-eight quiver. This is the directed graph with one vertex $\bullet$ and two arrows $\alpha$ and $\beta$ whose source and target are each that vertex.  As a category $FQ$ is isomorphic to the free monoid on two generators viewed as a category.  The graph homomorphism $Q\to U\mathscr C$ given by $\bullet\mapsto \lbrace x.\top\rbrace$ and $\alpha\mapsto [X(x)=y]$, $\beta\mapsto [ Y(x)=y]$ extends to a \emph{bona fide} functor $\rho\colon FQ\to \mathscr C$ whose action on arrows $w\in FQ$ is $\rho(w) = [w(x)=y]\colon \lbrace x.\top\rbrace\to\lbrace y.\top\rbrace$ by Lemma 2.3.  The free monoid $FQ$ admits an epimorphism $e\colon FQ\to N$.  The claim is that there is a functor $N\to \mathscr C_{\mathbb T[\phi]}$ making a commutative diagram
$$\begin{tikzpicture}
\node(1){$FQ$};
\node(2)[node distance=1in, right of=1]{$\mathscr C$};
\node(3)[node distance=.8in, below of=1]{$N$};
\node(4)[node distance=1in, right of=3]{$\mathscr C_{\mathbb T[\phi]}$};
\draw[->](1) to node [above]{$\rho$}(2);
\draw[->](2) to node [right]{$\iota$}(4);
\draw[->>](1) to node [left]{$e$}(3);
\draw[->,dashed](3) to node [below]{$\bar\rho$}(4);
\end{tikzpicture}$$  
where $\iota$ denotes the inclusion described in $\S$2.3.  The desired $\bar\rho$ is given by $\iota\circ\rho$ provided it preserves the relations giving $N$.  It suffices to show that $\iota\circ\rho(u_i) = \iota\circ\rho(v_i)$ since these are functors.  But the equation is equivalent to the statement that $[f_i(y)=z]=[g_i(y) = z]$ holds in $\mathscr C_{\mathbb T[\phi]}$.  That the equations $f_i(y)=z$ and $g_i(y)=z$ are provably equivalent in $\mathbb T[\phi]$ is shown using the cut rule.  One derivation is
\[ \dfrac{\dfrac{\dfrac{}{\top\vdash \phi[y/x]}}{f_i(y)=z\vdash \phi[y/x]\wedge f_i(y)=z}\qquad \dfrac{}{\phi[y/x] \wedge f_i(y)=z\vdash g_i(y)=z}}{f_i(y)=z\vdash g_i(y)=z}
\]
using transitivity of equality on the right and substitution from the axiom $\top\vdash\phi$ on the upper-left.  The other derivation is analogous.  Thus, $\bar\rho$ is well-defined.  The square commutes by definition.  \end{proof}

The lemma allows us to mimic the undecidability proof of W. Baur \cite{Baur} that was later adapted by M. Prest to show that $k\langle X,Y\rangle$ has an undecidable classical first-order theory of modules.  Ultimately, however, our methods avoid the technicalities of ``extension by definitions'' used in both proofs. 

\begin{theo}  The cartesian theory $\mathbb T := \mathbb T_{k\langle X,Y\rangle}$ is undecidable.
\end{theo}
\begin{proof}  Throughout use the notation from the discussion above.  In particular, $u$ and $v$ are any fixed words in $a$ and $b$.  The following statements (1) and (2) will be seen to be equivalent.
\begin{equation} \text{The words}\;u\;\text{and} \; v\;\text{over $\lbrace a,b\rbrace$ are equivalent as a result of the relations defining $N$}.
\end{equation}
\begin{equation} \text{The sequent}\;\bigwedge_{i=1}^r (f_{i}(x) = g_{i}(x)) \vdash_{x} f(x) =g(x)\;\text{is provable}.
\end{equation}
Again let $\phi$ denote the antecedent of (2); and throughout let $\mathscr C$ denote the syntactic category of $\mathbb T_{k\langle X,Y\rangle}$.  That (1) implies (2) is obtained from the preceding lemma.  For if $u=v$ in $N$, it follows that $\bar\rho(u) = \bar\rho(v)$ holds in $\mathscr C_{\mathbb T[\phi]}$ so that $\phi\vdash f(x)= g(x)$ is therefore provable in $\mathbb T$ by Lemma 2.8.

Now, assume that (1) fails; that is, neither $u$ nor $v$ can be obtained from the other by a finite sequence of replacements.  The claim is that there is a model in which (2) fails.  To construct this, let $k[N]$ denote the monoid-algebra on $N$.  This is the set of functions $\alpha\colon N\to k$ taking only finitely many nonzero values, viewed as an algebra containing copies of both $N$ and $k$ as in \cite{Lang}.  In particular elements of $u\in N$ are identified with Kroenecker-$\delta$ functions, namely, $\delta_u\colon N\to k$ taking the value $1\in k$ at $u$ but $0$ at all others.  Since $\delta_u\delta_v=\delta_{uv}$ holds, it follows that $N$ acts on $k[N]$ by the ring multiplication $u\alpha :=\delta_u\alpha$. 

Given $u,v\in N$, the action has the property that if $u\alpha=v\alpha$ for all $\alpha\in k[N]$, then $u=v$.  Thus, the assumption that $u\neq v$ in $k[N]$ means that there is $\alpha_0\in k[N]$ with $u\alpha_0\neq v\alpha_0$.  But $u_i\alpha_0 = v_i\alpha_0$ holds for all $i$ since $u_i$ and $v_i$ are equal in $k[N]$.  But this essentially proves that (2) must fail.  For $k[N]$ is a module over $k\langle X,Y\rangle$, hence a model, because there is an algebra homomorphism $k\langle X,Y\rangle\to k[N]$.  Thus, $\alpha_0$ yields an element of the subobject interpreting the antecedent of (2) in the model.  But $\alpha_0$ is not in the subobject interpreting the consequent, as $f(x_0)\neq g(x_0)$.  Therefore, $[[x.\phi ]]_{k[N]} \leq [[x.f(x) = g(x)]]_{k[N]}$ does not hold in $k[N]$.  Thus, (2) is not provable, by soundness.  This shows that (2) implies (1).

There is thus a class of sequents of the form of (2) above such that any such sequent is provable if, and only if, a corresponding equivalence $u=v$ holds.  Thus, if the set of sequents of the form (2) were decidable, the word problem for $N$ would be as well.  Therefore, $\mathbb T_{k\langle X,Y\rangle}$ must be undecidable.  \end{proof}

As a remark, it should be noted that the proof given above appears to work in Boolean first-order categorical logic with one change.  The sequent (2) should be replaced by 
\[\bigwedge_{i=1}^r \forall x(f_{i}(x) = g_{i}(x)) \vdash_{x} \forall x(f(x) =g(x)).
\]
This is the sequent used in \cite{Baur}.  This indicates that our technique recaptures the original undecidability result.

\section{Representation Embeddings and ``Wild Implies Undecidable''}

\subsection{Indecomposable Projective Cartesian Copresheaves}

An object $p$ of a category $\mathscr C$ is projective if given an epimorphism $f\colon c\to p$ there is a section $s\colon p\to c$ such that $fs=1_p$.  In this situation the object $p$ is a retract of $c$.  For the main theorem, finitely-generated projective models of a cartesian theory need to be characterized precisely.  But in fact finitely-generated projectives of $\mathbf{Cart}(\mathscr C,\mathbf{Set})$ can be characterized whenever $\mathscr C$ is small, cartesian and Cauchy-complete in the sense that every idempotent splits.  The following lemma gives examples of such categories.

\begin{lemma} For any cartesian theory, $\mathscr C_\mathbb{T}$ is Cauchy-complete. 
\end{lemma}
\begin{proof}  Use the property D1.4.4(i) of \cite{elephant}, namely, that $[\theta]\colon\lbrace \mathbf{x}.\phi\rbrace\to \lbrace \mathbf{y}.\psi\rbrace$ is an isomorphism if, and only if, $\theta$ is also $\mathbb T$-provably functional from $\lbrace \mathbf{y}.\psi\rbrace$ to $\lbrace \mathbf{x}. \phi\rbrace$.  So, any idempotent $[\theta]\colon \lbrace \mathbf{x}.\phi\rbrace\to \lbrace \mathbf{x}.\phi\rbrace$ is an isomorphism; but any idempotent automorphism is necessary the identity, which splits.  \end{proof}

Now, whether or not $\mathscr C$ is Cauchy-complete, that representable functors are projective can be proved using the naturality of Yoneda's Lemma.  Any projective of $[\mathscr C,\mathbf{Set}]$ remains projective in the subcategory $\mathbf{Cart}(\mathscr C,\mathbf{Set})$.  And it has been observed that representables are finitely-generated in $\mathbf{Cart}(\mathscr C,\mathbf{Set})$.  Thus, every representable is finitely-generated projective in $\mathbf{Cart}(\mathscr C,\mathbf{Set})$ for $\mathscr C$ small.  The question is now whether the finitely-generated projectives of $\mathbf{Cart}(\mathscr C,\mathbf{Set})$ can be characterized.  The following gives the answer when $\mathscr C$ is cartesian and Cauchy-complete.  Its proof is based on that of 5.22 in \cite{AR}.

\begin{theo} If $\mathscr C$ is small, cartesian and Cauchy-complete, then each finitely-generated projective of $\mathbf{Cart}(\mathscr C,\mathbf{Set})$ is representable.
\end{theo}
\begin{proof}
Let $P\colon \mathscr C\to\mathbf{Set}$ denote a cartesian copresheaf.  This is canonically a filtered colimit of representables as in \S 1.2.  The colimit can be taken in $\mathbf{Cart}(\mathscr C,\mathbf{Set})$.  Take one of the canonical maps $\mathbf y(c)\to P$ from the colimit.  Denote this by $\tilde x$ with $x$ corresponding to the canonical map under the Yoneda correspondence.  This map factors as an epimorphism followed by a monomorphism
$$\begin{tikzpicture}
\node(1){$\mathbf y(c)$};
\node(2)[node distance=.7in, right of=1]{$$};
\node(3)[node distance=.7in, right of=2]{$P$};
\node(4)[node distance=.6in, below of=2]{$I_{c,x}$};
\draw[->](1) to node [above]{$\tilde x$}(3);
\draw[->](1) to node [left]{$$}(4);
\draw[>->,dashed](4) to node [below]{$\;\;\;\;\;\;\;m_{c,x}$}(3);
\end{tikzpicture}$$
using the canonical image copresheaf.  The images $I_{c,x}$ are cartesian.  Thus, together over all pairs $(c,x)$, these images give another diagram in $\mathbf{Cart}(\mathscr C,\mathbf{Set})$ indexed by the opposite of the category of elements.  It then follows by uniqueness of colimits that $P\cong \lim_{\to} I_{c,x}$. 

This shows that $P$ is a filtered union.  That $P$ is finitely generated in $\mathbf{Cart}(\mathscr C,\mathbf{Set})$ implies that the canonical map
\[ \lim_{\to} \mathrm{Hom}(P,I_{c,x})\longrightarrow \mathrm{Hom}(P,P) 
\]
is an isomorphism in $\mathbf{Set}$.  Now, the colimit on the left is a quotient of the disjoint union of the $\mathrm{Hom}(P,I_{c,x})$ as in 2.13.3 of \cite{Handbook1}.  Thus, the isomorphism above yields an arrow $f\colon P\to I_{c,x}$ for some $(c,x)$ making $m_{c,x}f=1_P$.  This shows that $m_{c,x}$ is an isomorphism.  So, the corresponding representable $\mathbf y(c)$ admits an epimorphism to $P$.  So, if $P$ is projective, then $P$ is a retract of $\mathbf y(c)$.  Retracts of representables are representable when $\mathscr C$ is Cauchy complete as in A1.1.10 of \cite{elephant} and 6.5.6 of \cite{Handbook1}.  
\end{proof}

\begin{cor}  For $\mathscr C$ small and Cauchy-complete, $\mathscr C^{op}$ is (weakly) equivalent to the full subcategory of $\mathbf{Cart}(\mathscr C,\mathbf{Set})$ of finitely-generated projectives.
\end{cor}
\begin{proof}
The Yoneda embedding is fully faithful.  Thus, it identifies $\mathscr C^{op}$ as weakly equivalent to the full subcategory of $[\mathscr C,\mathbf{Set}]$ of representables, hence to the full subcategory of $\mathbf{Cart}(\mathscr C,\mathbf{Set})$ of finitely-generated projectives by Theorem 4.2.  This weak equivalence can be made into a strong equivalence using the Axiom of Choice.  By fixing a representation for each representable functor, any morphism between two will be of the form $\mathscr C(h,-)$ for some arrow $h$ of $\mathscr C$ as in the corollary to Yoneda's Lemma in \cite{MacLane}.
\end{proof}

\subsection{Main Theorem}

Let $\mathbb T$ denote a cartesian theory with $\mathbb T$-$\mathbf{Mod}(\mathscr D)$ its category of models in a cartesian category $\mathscr D$.  The previous results show that since the syntactic category is Cauchy-complete, its opposite is equivalent to the full subcategory of finitely-generated projectives models.

\begin{define} A representation embedding of cartesian theories $\mathbb T_1$ and $\mathbb T_2$ is a functor
\[ E\colon \mathbb T_1\text -\mathbf{Mod}(\mathbf{Set})\to\mathbb T_2\text{-}\mathbf{Mod}(\mathbf{Set})
\]
that preserves finitely-generated projective models, and that both preserves and reflects epimorphisms between objects of the full subcategory of finitely-generated projective models.
\end{define}

This recalls standard definitions in \cite{Simson} and \cite{Benson}, but is phrased in a degree of generality appropriate for what can be proved in Theorem 4.5.  And indeed a wild algebra, as described in $\S$1, does come with a representation embedding in this sense.  For the functor $M\otimes_{k\langle X,Y\rangle}-$ has $M$ finitely-generated and free over $k\langle X,Y\rangle$.  Thus, it preserves finitely-generated modules and finitely-generated projectives in particular.  It is also faithfully flat, thus preserving and reflecting all epimorphisms. 

\begin{theo}  A representation embedding of cartesian theories $E\colon \mathbb T_1\to \mathbb T_2$ induces a functor of syntactic categories $T\colon \mathscr C_{\mathbb T_1}\to\mathscr C_{\mathbb T_2}$ that preserves and reflects provability.
\end{theo}
\begin{proof}  The proof is largely indicated by the following diagram.  On the right $\Phi$ and $\Psi$ name the equivalences as in Lemma 2.2.  And the syntactic categories are identified via the Yoneda embeddings as full subcategories of the respective categories of cartesian functors as in Corollary 4.3.
$$\begin{tikzpicture}
\node(1){$\mathscr C_{\mathbb T_1}^{op}$};
\node(2)[node distance=1.5in, right of=1]{$\mathbf{Cart}(\mathscr C_{\mathbb T_1},\mathbf{Set})$};
\node(3)[node distance=.8in, below of=1]{$\mathscr C_{\mathbb T_2}^{op}$};
\node(4)[node distance=1.5in, right of=3]{$\mathbf{Cart}(\mathscr C_{\mathbb T_2},\mathbf{Set})$};
\node(5)[node distance=1.5in, right of=2]{$\mathbb T_1\text{-}\mathbf{Mod}(\mathbf{Set})$};
\node(6)[node distance=1.5in, right of=4]{$\mathbb T_2\text{-}\mathbf{Mod}(\mathbf{Set}).$};
\draw[->](1) to node [above]{$\mathbf y$}(2);
\draw[->](5) to node [right]{$E$}(6);
%\draw[->,dashed](1) to node [left]{$$}(3);
\draw[->](3) to node [below]{$\mathbf y$}(4);
\draw[->](2) to node [above]{$\Phi$}(5);
\draw[<-](4) to node [below]{$\Psi$}(6);
\end{tikzpicture}$$
Since $E$ preserves finitely-generated projectives, the composite $\Psi\circ E\circ \Phi$ takes representables to representables.  So, there is an induced functor $\mathscr C_{\mathbb T_1}^{op}\to\mathscr C_{\mathbb T_2}^{op}$.  It preserves and reflects epimorphisms by the assumed properties of $E$.  And this is equivalent to giving a functor $T\colon \mathscr C_{\mathbb T_1}\to\mathscr C_{\mathbb T_2}$ preserving and reflecting monomorphisms.  Thus, $T$ preserves and reflects provability in the sense of \S 2.3 by Lemma 2.4..\end{proof}

\begin{cor}  Take the same set-up as that in 4.5.  If $\mathbb T_1$ is undecidable, then $\mathbb T_2$ is also undecidable.
\end{cor}
\begin{proof}  If $\mathbb T_2$ were decidable, then the induced functor $T$ of syntactic categories would provide an algorithm for $\mathbb T_1$, contradicting undecidability.  \end{proof}

There is now the following result, given as a corollary to the general situation.  This is the complete resolution of the reformulated ``wild implies undecidable'' conjecture.
\begin{cor}  Let $S$ denote a wild $k$-algebra.  The theory $\mathbb T_S$ is undecidable. 
\end{cor}
\begin{proof} As already observed, $M\otimes_{k\langle X,Y\rangle}-$ is a representation embedding.  Thus, if $\mathbb T_S$ were decidable, then $\mathbb T_{k\langle X,Y\rangle}$ would be as well, a contradiction of Theorem 3.2.\end{proof}

\section{Summary and Discussion}

Theorem 4.5 and its corollaries provide an affirmative resolution of a ``wild implies undecidable'' conjecture reformulated in cartesian logic.  It remains to see, however, whether these results settle the original conjecture.  A brief summary helps to illuminate the central questions.  

In the original formulation, it is the \emph{first-order theory} of the modules over a wild algebra that should be undecidable.  Finite dimensionality is part of the original conjecture, in that the wild algebra is taken to be finite-dimensional.  The idea is that the ``finite-dimensional representation theory'' of a wild algebra should be ``at least as complex as'' that of $k\langle X,Y\rangle$.  It has been proved here that the \emph{cartesian theory} of modules over $k\langle X,Y\rangle$ is undecidable; and that the undecidability is inherited by any module category admitting a different kind of embedding.  So, how much of a departure is the use of cartesian logic?  And does a resolution of a rephrasing of the conjecture in cartesian logic really settle the original?

Cartesian logic is the simplest fragment of first-order categorical logic that both axiomatizes the concrete algebraic data of set-theoretic modules, gives an undecidable theory, and allows use of syntactic categories and a completeness theorem.  Thus, there appears to be no simpler fragment allowing use of the crucial technology, but also no special need of the added complexity of a richer fragment of first-order logic.  In this sense, cartesian logic is exactly what is required.

 The intuitive meanings of the extra connectives of full first-order logic are captured in other features of cartesian logic anyway.  For example, the turnstile `$\vdash$' in the sequent notation acts much like an implication `$\to$'.  And universal quantification is covered in a sense too.  For a statement such as ``the sequent `$\phi\vdash \psi$' is provable'' means ``in any model, the extent to which $\psi$ is true is no less than the extent to which $\phi$ is true.''  As the interpretations are structured sets, the statement is equivalent to $[[\phi]]\subset [[\psi]]$ which has an implicit universal quantification. These extra properties of categorical semantics are precisely what allowed us to recapture the schema of Baur's original undecidability proof \cite{Baur}.  

But the use of cartesian logic has its own rewards.
  The ambient first-order logic of classical module theory is \emph{Boolean}.  The cartesian theories used here cannot even formulate the Boolean axiom schema, since there is neither negation nor implication nor falsity `$\bot$'.  Accordingly, the cartesian theory of $k\langle X,Y\rangle$-modules is non-classical, hence weaker than the classical first-order theory used in the traditional model theory of modules.  Thus, Theorem 3.2 is of independent interest in any event.

In terms of explanation, the extra power of full first-order logic is a red-herring.  For the constructive cartesian logic is enough to make the reason for the undecidability of $\mathbb T_{k\langle X,Y\rangle}$ clear.  That is, the syntactic categories of the cartesian theories interpret an undecidable word problem.  This should not be too surprising given the construction of the theories and the fact that these categories are generalized monoids.  The point is that the machinary of the category theory of cartesian logic gives life to the intuition.

The latter question is trickier and must be left ultimately to the specialists in representation theory.  But a few comments are nonetheless in order.  The obvious objection to the presentation here is simply that cartesian logic is not first-order.  So, strictly speaking, the original conjecture has not been resolved.  In a way, our replies to such a line of objection have been made above.  It should be pointed out, however, that since cartesian logic is a fragment of full first-order logic, that the cartesian theory of $k\langle X,Y\rangle$-modules is undecidable suggests that the first-order theory is too.  However, it is not clear that the same techniques pass undecidability along a representation embedding in the manner of the proof of Theorem 4.5.  The reason is just that the syntactic categories of first-order logic do not behave in the same manner as those for cartesian logic.  More work here needs to be done.

What is more interesting about the results of this paper is that no crucial use of indecomposability or finite-dimensionality has been made at all.  In fact the definition of a representation embedding $E$ includes the precise minimum conditions needed for the proof of Theorem 4.5 to work in the manner it does.  One could better respect the original definitions by asking, for example, that all finitely generated model are preserved.  But while the more general definition given here makes for clean, high-level conceptual proofs of the main results, it does not really explain the importance of indecomposability and finite-dimensionality in the original definitions and conjecture.  So, even if in some sense the results of this paper are taken to settle the ``wild implies undecidable'' conjecture, it is still unknown as to whether there is a compelling and more purely ``representation-theoretic reason'' why it is true.  

Now, indecomposability has been dropped from the main definition because it was not required to prove the theorems.  Finite-dimensionality has been dropped because it cannot be axiomatized in first-order logic and because the main result turns out to apply to algebras and modules of any dimension.  Strictly speaking, the resolution in Corollary 4.6 does not actually require the vector-space structure of representations, but only the action of the ring.  So, the crucial properties that have been used are the that $M$ is a finitely-generated bimodule and in particular free over $k\langle X,Y\rangle$.  Thus, the conditions identifying wild algebras could conceivably be weakened considerably to assert only the existence of a bimodule $M$ that is faithfully flat over $k\langle X,Y\rangle$ and whose tensor $M\otimes_{k\langle X,Y\rangle} -$ preserves finitely-generated projectives.  Wild algebras under this defintion would have undecidable cartesian theories of modules.  These somewhat delicate points, however, must be left to the true specialists in representation theory.

\section{Acknowledgements}

The research of this paper has been supported by NSGS funding through Dalhousie University, as well as by the NSERC Discovery Grant of Dr. Dorette Pronk at Dalhousie.  The author would like to thank Dr. Pronk for supervising this ongoing research and for reading several drafts of the paper; Dr. Peter Selinger at Dalhousie for some illuminating consultation; and Dr. Calin Chindris at the University of Missouri-Columbia for supervising the initial stages of the research and for his ongoing support and encouragement. 

\bibliography{research}
\bibliographystyle{alpha}

\end{document}